\documentclass{amsart}
\usepackage{proof}
\usepackage{dst}
\usepackage{ast2}

\usepackage[margin=1.4in]{geometry}

\newcommand{\abr}{\ensuremath{\, \usftext{r} \, }}
\newcommand{\aber}{\ensuremath{\, \usftext{e} \, }}
\newcommand{\HAP}{\ensuremath{ \usftext{HAP}  }}
\newcommand{\Hcap}{\ensuremath{ \usftext{H}  }}
\newcommand{\APP}{\ensuremath{ \usftext{APP}  }}
\newcommand{\EON}{\ensuremath{ \usftext{EON}  }}
\newcommand{\LPT}{\ensuremath{ \usftext{LPT}  }}
\newcommand{\QFAC}{\ensuremath{ \usftext{QF-AC}  }}

\usepackage{amssymb,amsmath,amsthm}
\usepackage[foot]{amsaddr}
\title{Arithmetical conservation results}
\author{Benno van den Berg$^1$}
\address{${}^1$ Institute for Logic, Language and Computation (ILLC), University of Amsterdam, P.O. Box 94242, 1090 GE Amsterdam, the Netherlands. E-mail: bennovdberg@gmail.com.}
\author{Lotte van Slooten$^2$}
\address{${}^2$ Mathematical Institute, Utrecht University, P.O. Box 80010, 3508 TA Utrecht, the Netherlands. E-mail: lottevanslooten@live.nl.}

\date{June 15, 2017}

\begin{document}

\begin{abstract}
In this paper we present a proof of Goodman's Theorem, a classical result in the metamathematics of constructivism, which states that the addition of the axiom of choice to Heyting arithmetic in finite types does not increase the collection of provable arithmetical sentences. Our proof relies on several ideas from earlier proofs by other authors, but adds some new ones as well. In particular, we show how a recent paper by Jaap van Oosten can be used to simplify a key step in the proof. We have also included an interesting corollary for classical systems pointed out to us by Ulrich Kohlenbach.
\end{abstract}

\maketitle

\section{Introduction}

The axiom of choice has a special status in constructive mathematics. On the one hand, it is arguably justified on the constructive interpretation of the quantifiers. Indeed, one could argue that a constructive proof of $\forall x \in X \, \exists y \in Y \, \varphi(x, y)$ should contain, implicitly, an effective method for producing, given an arbitrary $x \in X$, an element $y \in Y$ such that $\varphi(x, y)$. Such an effective method can then be seen as a constructive choice function $f: X \to Y$ such that $\varphi(x, f(x))$ holds for any $x \in X$. In fact, it is precisely for this reason that the type-theoretic axiom of choice is \emph{provable} in Martin-L\"of's constructive type theory (see \cite{martinlof84}).

On the other hand, many standard systems for constructive mathematics do not include the axiom of choice. One example of such a system is Aczel's constructive set theory {\bf CZF}. Indeed, an argument due to Diaconescu shows that in {\bf CZF} the set-theoretic axiom of choice {\bf AC} implies a restricted form of the Law of Excluded Middle (see, for example, \cite{aczelrathjen01}). One thing one learns from this is that the status of the axiom of choice may depend on the way it is formulated as well as on the background theory. (For an interesting perspective on these matters, see \cite{martinlof09}.)

In the present paper we concentrate on $\ha$, Heyting arithmetic in all finite types. This system dates back to the work by Kreisel from the late fifties \cite{kreisel59} and has since become important in the study of constructivism. Currently, it is also playing an essential r\^ole in the work on the extraction of programs from proofs and proof mining, as can be seen from the recent books \cite{schwichtenbergwainer12, kohlenbach08}. In addition, it is also starting to attract attention in the Reverse Mathematics community, as can be seen from some recent papers on higher-order reverse mathematics like \cite{kohlenbach05c,hunter08,schweber15}.

The precise formulation of the axiom of choice that we will look at is the following axiom of choice for all finite types:
\[ \AC: \quad \forall x^\sigma \, \exists y^\tau \, \varphi(x, y, \tup z) \to \exists f^{\sigma \to \tau} \, \forall x^\sigma \, \varphi(x, fx, \tup z). \]
This version of the axiom of choice is not provable in $\ha$; however, one can show that $\ha + \AC$ and $\ha$ are equiconsistent (for example, by using Kreisel's modified realizability from \cite{kreisel59}).

This is in marked contrast to what happens in the case of $\pa$, Peano arithmetic in all finite types. Indeed, using classical logic one can derive comprehension axioms from the axiom of choice: for if $\varphi(x^\sigma)$ is any formula in the language of $\pa$, then
\[ \forall x^\sigma \, \exists n^0 \, \big( \, n = 0 \leftrightarrow \varphi(x^\sigma) \, \big) \]
is derivable using classical logic, from which
\[ \exists f^{\sigma \to 0} \, \forall x^\sigma \, \big( \, f(x) = 0 \leftrightarrow \varphi(x^\sigma) \, \big) \]
follows using the axiom of choice. For this reason the system $\pa + \AC$ has the strength of full higher-order arithmetic, a much stronger system that $\pa$. This is another manifestation of the special status of the axiom of choice in constructivism.

In the constructive case, more is true. Not only are  $\ha + \AC$ and $\ha$ equiconsistent, but they also prove the same arithmetical sentences (where a sentence is arithmetical if its quantifiers range over natural numbers and the equalities it contains are between natural numbers). This is a classical result in the metamathematics of constructivism and goes by the name Goodman's Theorem. As the name suggests, it was first proved by Nicholas Goodman in 1976 (see \cite{goodman76}). The original proof was based on a rather complicated theory of constructions and after this proof was published, various people have sought simpler proofs. One such proof was given by Goodman himself using a new proof-theoretic interpretation combining ideas from forcing and realizability \cite{goodman78}. Beeson showed how this can be understood as the composition of forcing and realizability and extended Goodman's theorem to the extensional setting (see \cite{beeson79} and also \cite{beeson85}). Other proofs have been given by Gordeev \cite{gordeev88}, Mints \cite{mints75}, Coquand \cite{coquand13} and Renardel de Lavalette \cite{renardeldelavalette90}; the authors of this paper are unsure whether this list is complete. (An interesting observation, due to Kohlenbach, is that Goodman's Theorem can fail badly for fragments: see \cite{kohlenbach99}.)

What we have done in this paper is to give yet another proof of Goodman's Theorem. Our reasons for doing so are that we feel that despite its classic status, complete and rigorous proofs of this result are surprisingly rare, while some of the proofs that are complete are not the simplest or most transparent possible. What we have sought to do here is to give a proof which dots all the \emph{i}s. But we should stress that many of the ideas of our proof can already be found in the sources mentioned above. The main novelty may be in some of the details of the presentation and the observation that a recent paper by Jaap van Oosten (see \cite{vanoosten06}) can be used to simplify a key step in the proof. We have also included  an interesting corollary of Goodman's Theorem for classical systems pointed out to us by Ulrich Kohlenbach. We are grateful for his permission to include it here.

Like most of the proofs mentioned above, ours relies on a string of proof-theoretic interpretations starting from $\ha + \AC$ and ending with $\HA$, keeping the set of provable arithmetical sentences fixed. The string of interpretations is quite long: we could easily have made the proof shorter, but we felt that this would make the argument less transparent. Indeed, the proof combines many ideas and by making sure that each proof-theoretic interpretation relies on a single idea only, the whole structure of the argument becomes a lot easier to follow and far more intelligible.

In the intermediate stages we will make use of a theory of operations similar to Beeson's $\EON$ and Troelstra's $\APP$ (for which see \cite{beeson85,troelstravandalen88b,troelstra98}); our version of this is a system we have called $\HAP$. Systems of this form go back to the pioneering work of Feferman \cite{feferman75,feferman79} and we hope that with this paper we honour the memory of more than one great foundational thinker.

The contents of this paper are based on a Master thesis written by the second author and supervised by the first author \cite{slooten14}. Finally, we would like to thank the referee for a useful report.

\section{The systems $\HAP$ and $\HAP_\varepsilon$}

The aim of this section is to introduce the systems $\HAP$ and $\HAP_\varepsilon$ and show that they are conservative over $\HA$. Both these systems are formulated using the logic of partial terms $\LPT$ (also called $E^+$-logic), due to Beeson (see \cite{beeson85,troelstravandalen88a,troelstra98}). For the convenience of the reader we present an axiomatision following \cite{troelstra98}.

\subsection{Logic of partial terms $\LPT$} The idea of the logic of partial terms is that we want to have a logic in which we can reason about terms which do not necessarily denote (think Santa Claus or the present king of France). To express that a term $t$ denotes, or ``$t$ exists'', we will write $t \downarrow$. In fact, here we will consider this as an abbreviation for $t = t$.

Having terms around which do not denote, forces us to change the usual rules for the quantifiers. Where normally we can deduce $A[t/x]$ from $\forall x \, A$ for any term $t$, in $\LPT$ this is only possible if $t$ denotes; conversely, we can only deduce $\exists x \, A$ from $A[t/x]$ if $t$ denotes.

More precisely, the language of $\LPT$ is that of standard intuitionistic predicate logic with equality, with $t \downarrow$ as an abbreviation for $t = t$. Its axioms and inference rules are all substitution instances of:
\begin{displaymath}
\begin{array}{l}
\varphi \to \varphi \\
\varphi, \varphi \to \psi \Rightarrow \psi \\
\varphi \to \psi, \psi \to \chi \Rightarrow \varphi \to \chi \\
\varphi \land \psi \to \varphi, \varphi \land \psi \to \psi \\
\varphi \to \psi, \varphi \to \chi \Rightarrow \varphi \to \psi \land \chi \\
\varphi \to \varphi \lor \psi, \psi \to \varphi \lor \psi \\
\varphi \to \chi, \psi \to \chi \Rightarrow \varphi \lor \psi \to \chi \\
(\varphi \land \psi) \to \chi \Rightarrow \varphi \to (\psi \to \chi)  \\
\varphi \to (\psi \to \chi) \Leftrightarrow (\varphi \land \psi) \to \chi \\
\bot \to \varphi \\
\varphi \to \psi \Rightarrow \varphi \to \forall x \, \psi \quad ( x \not\in {\rm FV}(\varphi)) \\
\forall x \, \varphi \land t \downarrow \to \varphi[t/x] \\
\varphi[t/x] \land t \downarrow \to \exists x \, \varphi \\
\varphi \to \psi \Rightarrow \exists x \, \varphi \to \psi \quad (x \not\in {\rm FV}(\psi))
\end{array}
\end{displaymath}
For equality we have the following rules:
\begin{displaymath}
\begin{array}{l}
\forall x \, (\, x = x \,), \quad \forall x y \, ( \, x = y \to y = x), \quad \forall x y z \, ( \, x = y \land y = z \to x = z)  \\
\forall \tup{x} \tup{y} \, \big( \, \tup{x} = \tup{y} \land F(\tup{x}) \downarrow \to F(\tup{x}) = F({\tup y}) \,\big), \quad \forall \tup{x} \tup{y} \, ( \, R\tup{x} \land \tup{x} = \tup{y} \to R\tup{y})
\end{array}
\end{displaymath}
In addition, all the basic function and relation symbols will be assumed to be strict:
\[ c \downarrow, \quad F(t_1,\ldots, t_n) \downarrow \to t_i \downarrow, \quad R(t_1,\ldots,t_n) \to t_i \downarrow. \]
This includes equality:
\[ s = t \to s \downarrow \land t \downarrow. \]
It will be convenient to introduce the following weaker notion of equality:
\[ s \simeq t := (s \downarrow \lor t \downarrow) \to s = t. \]
So $s \simeq t$ expresses that $s$ and $t$ are equally defined and equal whenever defined. For this weaker notion of equality we can prove the following Leibniz schemata:
\[ \tup{s} \simeq \tup{t} \to F\tup{s} \simeq F\tup{t}, \quad s \simeq t \land A[t/x] \to A[s/x]. \]

\begin{rema}{addingdefpartialfunctions}
Below we will frequently exploit the following fact. Suppose $T$ is some theory based on the logic of partial terms and we can prove in $T$ that for some formula $\varphi(\tup x, y)$ we have
\[ \varphi(\tup x, y) \land \varphi(\tup x, y') \to y = y'. \]
Then we may extend $T$ to a theory $T'$ by introducing a new function symbol $f_\varphi$ and adding a new axiom
\[ \varphi(\tup x, y) \leftrightarrow y = f_\varphi(\tup x). \]
The resulting theory $T'$ will then be conservative over $T$. For a proof, see \cite[Section 2.7]{troelstravandalen88a}.
\end{rema}

\subsection{The system $\HAP$} In this paper a key r\^ole is played by a formal system which we will call $\HAP$. It is a minor variation on Beeson's $\EON$ and Troelstra's $\APP$ (for which see again \cite{beeson85,troelstravandalen88b,troelstra98}).

The language is single-sorted and the logic is based on $\LPT$. There are the usual arithmetical operations $0, S, +, \times$, with the usual axioms:
\begin{displaymath}
\begin{array}{l}
Sx \downarrow, x + y \downarrow, x \times y \downarrow, \\
Sx = Sy \to x = y, 0 \not= Sx, \\
x + 0 = x, x + Sy = S(x + y), \\
x \times 0 = x, x \times Sy = (x \times y) + y
\end{array}
\end{displaymath}
$\HAP$-formulas in this fragment of the language will be called \emph{arithmetical}. In addition, there will be an application operation written with a dot $\cdot$ and combinators
\begin{displaymath}
{\bf k}, {\bf s}, {\bf p}, {\bf p}_0, {\bf p}_1, {\bf succ}, {\bf r}
\end{displaymath}
Instead of $t_1 \cdot t_2$ we will often simply write $t_1 t_2$ and application associates to the left (so $t_1t_2t_3$ stands for $(t_1t_2)t_3$). Note that our assumption that all function symbols are strict implies that we have
\[ s \cdot t \downarrow \to s \downarrow \land t \downarrow. \]
The axioms for the combinators are
\begin{displaymath}
\begin{array}{l}
{\bf k}xy = x, {\bf s}xy \downarrow, {\bf s}xyz \simeq xz(yz), \\
{\bf p}_0 x \downarrow, {\bf p}_1x \downarrow, {\bf p}_0({\bf p}xy) = x, {\bf p}_1({\bf p}xy) = y, {\bf p}({\bf p}_0x)({\bf p}_1x) = x,\\
{\bf succ} \cdot x = Sx, {\bf r}xy0 = x, {\bf r}xy(Sz) = yz({\bf r}xyz).
\end{array}
\end{displaymath}
Finally, we have the induction scheme:
\[ \varphi[0/x] \land \forall x \, ( \varphi \to \varphi[Sx/x]) \to \forall x \, \varphi \]
for all $\HAP$-formulas $\varphi$.

\begin{rema}{comparisonwithEONandAPP}
The system $\HAP$ can be obtained from Troelstra's system $\APP$ (or Beeson's $\EON$) by making the following changes:
\begin{enumerate}
\item[(a)] $\APP$ has a unary predicate $N$ for being a natural number, which is dropped in $\HAP$. Indeed, in $\HAP$ every element acts as a natural number, in that the induction scheme is valid over the entire domain. For that reason $\HAP$ proves that equality is decidable and that there is an element $\bf e$ such that ${\bf e}xy = 0$ precisely when $x = y$.
\item[(b)] $\APP$ has an if-then-else construct $\bf d$ instead of a recursor $\bf r$. This is a minor difference, as these are interderivable (see \cite[Lemma 9.3.8]{troelstravandalen88b}).
\item[(c)] In $\HAP$ the arithmetical operations are primitive, allowing us to define the arithmetical fragment of $\HAP$. In addition, it allows us to change the interpretation of the application operation, while keeping the interpretation of the arithmetical fragment fixed (as in \reflemm{oraclesforHAP} below).
\end{enumerate}
\end{rema}

\begin{prop}{lambdaabstractionforHAP} For each term $t$ in the language of $\HAP$ and variable $x$, one can construct a term $\lambda x.t$, whose free variables are those of $t$ excluding $x$, such that $\HAP \vdash \lambda x.t \downarrow$ and $\HAP \vdash (\lambda x.t) x \simeq t$.
\end{prop}
\begin{proof} In case $t$ is a term built from variables, application and combinators, then there is a well-known abstraction algorithm (see, for example, \cite[Proposition 9.3.5]{troelstravandalen88b} and also \cite[p. 423]{troelstra98}). We define $\lambda x.t$ by induction on the structure of $t$. If $t$ is $x$ itself, then $\lambda x.x$ is ${\bf skk}$, while if $t$ is a variable or constant different from $t$, then $\lambda x.t$ is ${\bf k}t$. Finally, if $t = t_1t_2$, then $\lambda x.t$ is ${\bf s}(\lambda x.t_1)(\lambda x.t_2)$.

Of course, general terms in $\HAP$ can also be built using the arithmetical operations $S, +, \times$. However, since for ${\bf succ}, {\bf plus} := \lambda xy.{\bf r}x(\lambda uv.{\bf succ} \cdot v)y$ and ${\bf times} := \lambda xy.{\bf r}0(\lambda uv.{\bf plus}vy)y$ the system $\HAP$ proves
\[ {\bf succ} \cdot x = Sx, \quad {\bf plus} \cdot x \cdot y = x + y, \quad {\bf times} \cdot x \cdot y = x \times y, \]
any term is provably equal in $\HAP$ to one built purely from variables, the application operation and combinators.
\end{proof}

\begin{prop}{simpleconservativity}
The system $\HAP$ is conservative over $\HA$.
\end{prop}
\begin{proof} Indeed, we can interpret $\HAP$ in $\HA$ by exploiting the fact that one can develop basic recursion theory inside $\HA$. This allows one to interpret $x \cdot y$ as Kleene application: that is, $x \cdot y$ is the result of the partial recursive function coded by $x$ on input $y$ (whenever this is defined). More details can be found in, for instance, \cite[Proposition 9.3.12]{troelstravandalen88b}.
\end{proof}

\subsection{The system $\HAP_\varepsilon$} In the remainder of this section we will study an extension of $\HAP$. This extension, which we will call $\HAP_\varepsilon$, is obtained from $\HAP$ by adding for each arithmetical formula $\varphi(x,y)$ a constant $\varepsilon_\varphi$ as well as the following axioms:
\[ \exists y \, \varphi(x, y) \to \varepsilon_\varphi \cdot x \downarrow, \quad \varepsilon_\varphi \cdot x \downarrow \to \varphi(x, \varepsilon_\varphi \cdot x) \]
The goal of this subsection is to prove that the resulting system is still conservative over $\HA$.

\begin{prop}{conservativityofHAPEoverHAP}
Suppose $\psi(x, y)$ is a formula in the language of $\HAP$ and suppose $\HAP_f$ is the extension of $\HAP$ with a function symbol $f$ and the following axioms:
\[ \exists y \, \psi(x, y) \to f(x) \downarrow, \quad f(x) \downarrow \to \psi(x, f(x))  \]
Then $\HAP_f$ is conservative over $\HAP$.
\end{prop}
\begin{proof}
In view of \refrema{addingdefpartialfunctions} it suffices to show that we can conservatively add to $\HAP$ a relation symbol $F$ satisfying the formulas
\[ F(x, y) \land F(x, y') \to y = y', \quad  \exists y \, \psi(x, y) \to \exists y \, F(x, y), \quad F(x, y) \to \psi(x, y).  \]
Let us call this system $\HAP_F$. To show the conservativity of this extension we use forcing, with as forcing conditions finite approximations to the relation $F$. To be precise, a condition $p$ is a (coded) finite sequence of pairs \[ < (x_0, y_0), \ldots, (x_{n-1}, y_{n-1}) > \] such that $\psi(x_i, y_i)$ holds for any $i \lt n$ and all $x_i$ are distinct. We will write $q \leq p$ if $p$ is an initial segment of $q$, and $(x, y) \in p$ if there is some $i \lt n$ such that $(x, y) = (x_i, y_i)$.

For any $\HAP_F$-formula $\varphi$ we define a $\HAP$-formula $p \Vdash \varphi$ by induction on $\varphi$, as follows:
\begin{eqnarray*}
p \Vdash \varphi & := & \varphi \mbox{ if $\varphi$ is an atomic $\HAP$\mbox-formula} \\
p \Vdash F(x, y) &:= & (\forall q \leq p) \, (\exists r \leq q) \, (x, y) \in r \\
p \Vdash \varphi \land \psi & := & p \Vdash \varphi \land p \Vdash \psi \\
p \Vdash \varphi \lor \psi &:= & (\forall q \leq p) \, (\exists r \leq q) \, \big( \, r \Vdash \varphi \lor r \Vdash \psi \, \big) \\
p \Vdash \varphi \to \psi & := & (\forall q \leq p) \, ( \, q \Vdash \varphi \to q \Vdash \psi) \\
p \Vdash \forall x \, \varphi(x)  &:= & (\forall x) \, (\forall q \leq p)  \, q \Vdash \varphi(x) \\
p \Vdash \exists x \, \varphi(x) & := & (\forall q \leq p) \, (\exists r \leq q) \, (\exists x) \, r \Vdash \varphi(x).
\end{eqnarray*}
Note that we have \begin{displaymath}
\begin{array}{l}
\HAP \vdash p \leq q \land q \Vdash \varphi \to p \Vdash \varphi \mbox{ and } \\
\HAP \vdash \big( \, (\forall q \leq p) \, (\exists r \leq q) \, r \Vdash \varphi \, \big) \, \to p \Vdash \varphi
\end{array}
\end{displaymath} for all $\HAP_F$-formulas $\varphi$ and \begin{equation} \label{forcingabsolute} \HAP \vdash \big( \, p \Vdash \varphi \, \big) \leftrightarrow \varphi \end{equation} if $\varphi$ is a $\HAP$-formula. The idea now is to prove
\[ \HAP_F \vdash \varphi \Longrightarrow \HAP \vdash p \Vdash \varphi \]
by induction on the derivation of $\varphi$ in $\HAP_F$. We leave the verification of the $\HAP$-axioms to the reader and only check that the interpretations of the axioms we have added to $\HAP$ are provable in $\HAP$; for this we reason in $\HAP$.

Suppose $q \Vdash F(x, y)$ and $q \Vdash F(x, y')$. Then there is some $r \leq q$ such that $(x, y) \in r$ and $(x, y') \in r$. Because $r$ is a condition, we must have $y = y'$ and therefore $q \Vdash y = y'$. We  conclude that we have $p \Vdash F(x, y) \land F(x, y') \to y = y'$ for every condition $p$.

Suppose $p' \Vdash \exists y \, \psi(x, y)$. Our aim is to show $p' \Vdash \exists y \, F(x, y)$, so suppose $q' \leq p'$. Then there must be some $r \leq q'$ and some $y'$ such that $r \Vdash \psi(x,y')$. Now there are two possibilities: either there is some $y$ such that $(x, y) \in r$, or no such $y$ exists (we are using here that equality in $\HAP$ is decidable). In the former case we have $r \Vdash F(x, y)$; in the latter, we use  (\ref{forcingabsolute}) to deduce that $\psi(x, y')$ holds. We extend $r$ by appending the pair $(x, y')$ to obtain some new condition $r' \leq r$. For this condition $r'$ we have $r' \Vdash F(x, y')$. So in both cases there is some condition $r' \leq q'$ and some $y'$ such that $r' \Vdash F(x, y')$. We conclude that $p' \Vdash \exists y \, F(x, y)$ and hence that we have $p \Vdash \exists y \, \psi(x, y) \to \exists y \, F(x, y)$ for every condition $p$.

Finally, if $q \Vdash F(x, y)$, then $(x, y) \in r$ for some $r \leq q$. Since $r$ is a condition, we have $\psi(x, y)$ and hence $q \Vdash \psi(x, y)$ by (\ref{forcingabsolute}). So $p \Vdash F(x, y) \to \psi(x, y)$ for every condition $p$.
\end{proof}

\begin{lemm}{oraclesforHAP}
Suppose $\psi(x, y)$ is an arithmetical formula in the language of $\HAP$ and suppose $\HAP_{\bf f}$ is the extension of $\HAP$ with a constant ${\bf f}$ and the following axioms:
\[ \exists y \, \psi(x, y) \to {\bf f} \cdot x \downarrow, \quad {\bf f} \cdot x \downarrow \to \psi(x, {\bf f} \cdot x)  \]
Then $\HAP_{\bf f}$ is conservative over $\HA$.
\end{lemm}
\begin{proof}
The idea is to work in the system $\HAP_f$ from \refprop{conservativityofHAPEoverHAP} and redefine the application in $\HAP_f$ in such a way that we can use the partial function $f$ as an oracle. This has the desired effect of making the function $f$ representable. How this can be done is worked out in \cite[Theorem 2.2]{vanoosten06}. Let us just recall from this paper how one redefines the application.

For any $a, b$ an \emph{$f$-dialogue} between $a$ and $b$ is defined to be a code of a sequence $u = < u_0, \ldots, u_{n-1}>$ such that for all $i \lt n$ there is a $v_i$ such that
\[ a \cdot (<b> * u^{\lt i}) = {\bf p} \bot v_i \quad \mbox{ and } f(v_i) = u_i. \]
We say that the new application $a \cdot_f b$ is defined with value $c$ if there is an $f$-dialogue $u$ between $a$ and $b$ such that
\[ a \cdot (<b> * u) = {\bf p}\top c. \]
Here $<b>$ is the sequence consisting only of $b$, * stands for concatenation and $u^{\lt i}$ denotes $< u_0, \ldots, u_{i-1} >$ whenever $u$ codes some finite sequence. In addition, $\bot$ and $\top$ are assumed to be some choice for the booleans for which there is an if-then-else construct ${\bf d}$ such that ${\bf d}\top xy = x$ and ${\bf d}\bot xy = y$ (for example, $\bot = \lambda xy.y$ and $\top = {\bf k}$ and ${\bf d} = \lambda xyz.xyz$.)

To correctly interpret the $\HAP$-axioms the interpretation of the combinators needs to change as well (again, see \cite[Theorem 2.2]{vanoosten06}), but, crucially, the interpretations of the arithmetical operations can remain the same. Therefore the interpretation of $\psi$ is unaffected and the system remains conservative over $\HA$.
\end{proof}

\begin{theo}{HAPepsilonconservative}
The system $\HAP_\varepsilon$ is conservative over $\HA$.
\end{theo}
\begin{proof}
It suffices to prove that for each finite set of formulas $\varphi_1,\ldots,\varphi_n$ adding the $\varepsilon_{\varphi_i}$ together with their axioms is conservative over $\HAP$, since each proof in $\HAP$ uses only finitely many symbols and finitely many axioms. In fact, by considering
\[ \varphi(x, y) := \bigwedge_{i=1}^n \, \big( \, {\bf p}_0x = i \to \varphi_i({\bf p}_1x, y) \, \big) \]
one sees that it suffices to prove this for extensions with a single combinator $\varepsilon_\varphi$ only. But the statement that adding a single combinator of that form results in system conservative over $\HA$ is precisely \reflemm{oraclesforHAP}.
\end{proof}

\section{Realizability for $\HAP$}

Following Feferman \cite{feferman75,feferman79} we define abstract realizability interpretations of $\HAP$ into itself. In this section it will be convenient to regard disjunction as a defined connective, as follows:
\[ \varphi \lor \psi := \exists n \, ( \, (n = 0 \to \varphi) \land (n \not= 0 \to \psi) \, ). \]

\begin{defi}{realizabilityforAPP} (Feferman) For each $\HAP$-formula $\varphi$ we define a new $\HAP$-formula $x \abr \varphi$ (``$x$ realizes $\varphi$'') by induction on the structure of $\varphi$ as follows:
\begin{eqnarray*}
x \abr \varphi & := & \varphi \quad \mbox{ if } \varphi \mbox{ is atomic } \\
x \abr (\varphi \land \psi) & := & {\bf p}_0x \abr \varphi \land {\bf p}_1x \abr \psi \\
x \abr (\varphi \to \psi) & := & \forall y \, ( y \abr \varphi \to x \cdot y \downarrow \land \, x \cdot y \abr \psi)  \\
x \abr \forall y \, \varphi & := & \forall y \, ( \, x \cdot y \downarrow \land \, x \cdot y \abr \varphi) \\
x \abr \exists y \, \varphi & := & {\bf p}_1 x \abr \varphi[{\bf p}_0x/y]
\end{eqnarray*}
\end{defi}

\begin{theo}{soundnessforabstractrealizability}
If $\HAP \vdash \varphi(\tup y)$, then $\HAP \vdash \exists x \, \forall \tup y \, \big( \, x \cdot \tup y \abr \varphi(\tup y) \, \big)$.
\end{theo}
\begin{proof}
See, for example, \cite[Theorem VII.1.5]{beeson85}.
\end{proof}

In what follows we will also need an extensional variant of this abstract form of realizability. In this form of realizability the collection of realizers of a fixed formula $\varphi$ carries an equivalence relation, the intuition being that $x$ and $x'$ are equivalent if ``$x$ and $x'$ are identical as realizers of $\varphi$''. Crucially, realizers of an implication $\varphi \to \psi$ are required to send realizers which are equal as realizers of $\varphi$ to realizers which are equal as realizers of $\psi$.

\begin{defi}{extrealizabilityforAPP} (See \cite[Definition 6.1]{troelstra98} and \cite{vanoosten97}.) For each $\HAP$-formula $\varphi$ we define new $\HAP$-formulas $x \aber \varphi$ (``$x$ extensionally realizes $\varphi$'') and $x = x' \aber \varphi$ (``$x$ and $x'$ are identical as extensional realizers of $\varphi$'') by simultaneous induction on the structure of $\varphi$ as follows:
\begin{eqnarray*}
x \aber \varphi & := & \varphi \quad \mbox{ if } \varphi \mbox{ is atomic } \\
x = x' \aber \varphi & := & \varphi \land x = x' \quad \mbox{ if } \varphi \mbox{ is atomic } \\
x \aber (\varphi \land \psi) & := & {\bf p}_0 x \aber \varphi \land {\bf p}_1x \aber \psi \\
x = x' \aber (\varphi \land \psi) & := & {\bf p}_0x = {\bf p}_0x' \aber \varphi \land {\bf p}_1 x =  {\bf p}_1x' \aber \psi \\
x \aber (\varphi \to \psi) & := & \forall y, y' \, ( \, y = y' \aber \varphi \to x \cdot y \downarrow \land x \cdot y' \downarrow \land \, x \cdot y = x \cdot y' \aber \psi \, )\\
x = x' \aber (\varphi \to \psi) & := & x \aber (\varphi \to \psi) \land x' \aber (\varphi \to \psi) \land \forall y \, ( y \aber \varphi \to x \cdot y = x' \cdot y \aber \psi \, ) \\
x \aber \forall y \, \varphi & := & \forall y \, ( \, x \cdot y \downarrow \land \, x \cdot y \aber \varphi \, )\\
x = x' \aber \forall y \, \varphi & := & x \aber \forall y \, \varphi \land x' \aber \forall y \, \varphi \land \forall y \, ( \, x \cdot y = x' \cdot y \aber \varphi) \\
x \aber \exists y \, \varphi & := & {\bf p}_1x \aber \varphi[{\bf p}_0x/y] \\
x = x' \aber \exists y \, \varphi & := & {\bf p}_1 x = {\bf p}_1x' \aber \varphi[{\bf p}_0x/y] \land {\bf p}_0x = {\bf p}_0x'
\end{eqnarray*}
\end{defi}
One shows by induction on the structure of $\varphi$ that provably in $\HAP$ the relation $x = y \aber \varphi$ is symmetric and transitive and $x \aber \varphi$ is equivalent to $x = x \aber \varphi$. In addition, we have:

\begin{theo}{soundnessforabstractextrealizability}
If $\HAP \vdash \varphi(\tup y)$, then $\HAP \vdash \exists x \, \forall \tup y \, \big( \, x \cdot \tup y \aber \varphi(\tup y) \, \big)$.
\end{theo}
\begin{proof}
Routine.
\end{proof}

For the main theorem of this section we return to the system $\HAP_\varepsilon$.
\begin{theo}{APPepsilonselfrealizing}
The system $\HAP_\varepsilon$ proves $\varphi \leftrightarrow \exists x \, x \abr \varphi \leftrightarrow \exists x \, x \aber \varphi$ for every arithmetical formula $\varphi$.
\end{theo}
\begin{proof}
In $\HAP_\varepsilon$ we can define for every arithmetical formula $\varphi$ with free variables $\tup x$ a ``canonical realizer'' $j_\varphi$, as follows:
\begin{eqnarray*}
j_\varphi & := & \lambda \tup x. 0 \mbox{ if } \varphi \mbox{ is atomic and arithmetical } \\
j_{\varphi \land \psi} & := & \lambda \tup x.{\bf p}(j_\varphi
\cdot \tup x )(j_\psi \cdot \tup x) \\
j_{\varphi \to \psi} & := & \lambda \tup x. \lambda y. (j_\psi \cdot \tup x) \\
j_{\forall y \, \varphi } & := & \lambda \tup x. \lambda y. j_{\varphi} \cdot \tup x \cdot y \\
j_{\exists y \, \varphi } & := & \lambda \tup x. {\bf p}(\varepsilon_\varphi \cdot \tup x)(j_\varphi \cdot \tup x \cdot (\varepsilon_\varphi \cdot \tup x))
\end{eqnarray*}
One can now prove \[ \HAP_\varepsilon \vdash \varphi \leftrightarrow \exists x \, x \abr \varphi \leftrightarrow j_\varphi \cdot \tup x \abr \varphi, \]
as well as
 \[ \HAP_\varepsilon \vdash \varphi \leftrightarrow \exists x \, x \aber \varphi \leftrightarrow j_\varphi \cdot \tup x \aber \varphi, \]
by induction on $\varphi$, assuming that $\tup x$ lists all free variables in $\varphi$.
\end{proof}

\begin{coro}{selfrealizingarithmeticalconservative}
Let $\Hcap$ be either $\HAP$ plus the schema $\varphi \leftrightarrow \exists x \, x \abr \varphi$ for all arithmetical $\varphi$, or $\HAP$ plus the schema $\varphi \leftrightarrow \exists x \, x \aber \varphi$ for all arithmetical $\varphi$. Then $\Hcap$ is conservative over $\HA$.
\end{coro}
\begin{proof}
This follows from \reftheo{HAPepsilonconservative} and \reftheo{APPepsilonselfrealizing}.
\end{proof}

\section{Applications to systems in higher types}

In this section we will discuss applications to systems in higher types, in particular, Goodman's Theorem. Various versions of finite-type arithmetic exist and the differences tend to be subtle, so first we will explain the precise version we will be working with.

Our starting point is the system $\ha$ from \cite[pages 444-449]{troelstravandalen88b}. This is a system formulated in many-sorted intuitionistic logic, where the sorts are the finite types.
\begin{defi}{finitetypes} The \emph{finite types} are defined by induction as follows: 0 is a finite type, and if $\sigma$ and $\tau$ are finite types, then so are $\sigma \to \tau$ and $\sigma \times \tau$. The type 0 is the \emph{ground} or \emph{base type}, while the other types will be called \emph{higher types}.
\end{defi}
There will be infinitely many variables of each sort. In addition, there will be constants:
\begin{enumerate}
\item for each pair of types $\sigma, \tau$ a combinator ${\bf k}^{\sigma, \tau}$ of sort $\sigma \to (\tau \to \sigma)$.
\item for each triple of types $\rho, \sigma, \tau$ a combinator ${\bf s}^{\rho, \sigma, \tau}$ of type $(\rho \to (\sigma \to \tau)) \to ((\rho \to \sigma) \to (\rho \to \tau))$.
\item for each pair of types $\rho, \sigma$ combinators ${\bf p}^{\rho, \sigma}, {\bf p}^{\rho, \sigma}_0, {\bf p}^{\rho, \sigma}_1$ of types $\rho \to (\sigma \to \rho \times \sigma)$, $\rho \times \sigma \to \rho$ and $\rho \times \sigma \to \sigma$, respectively.
\item a constant 0 of type 0 and a constant $S$ of type $0 \to 0$.
\item for each type $\sigma$ a combinator ${\bf R}^\sigma$ (``the recursor'') of type $\sigma \to ((0 \to (\sigma \to \sigma)) \to (0 \to \sigma))$.
\end{enumerate}

\begin{defi}{terms}
The terms of $\ha$ are defined inductively as follows:
\begin{itemize}
\item each variable or constant of type $\sigma$ will be a term of type $\sigma$.
\item if $f$ is a term of type $\sigma \to \tau$ and $x$ is a term of type $\sigma$, then $fx$ is a term of type $\tau$.
\end{itemize}
\end{defi}
The convention is that application associates to the left, which means that an expression like $fxyz$ has to be read as $(((fx)y)z)$.

\begin{defi}{formulas}
The formulas of $\ha$ are defined inductively as follows:
\begin{itemize}
\item $\bot$ is a formula and if $s$ and $t$ are terms of the same type $\sigma$, then $s =_\sigma t$ is a formula.
\item if $\varphi$ and $\psi$ are formulas, then so are $\varphi \land \psi, \varphi \lor \psi, \varphi \to \psi$.
\item if $x$ is a variable of type $\sigma$ and $\varphi$ is a formula, then $\exists x^\sigma \, \varphi$ and $\forall x^\sigma \, \varphi$ are formulas.
\end{itemize}
\end{defi}

Finally, the axioms and rules of $\ha$ are:
\begin{enumerate}
\item[(i)] All the axioms and rules of many-sorted intuitionistic logic (say in Hilbert-style).
\item[(ii)] Equality is an equivalence relation at all types:
\[ x = x, \qquad x = y \to y = x, \qquad x = y \land y = z \to x = z \]
\item[(iii)] The congruence laws for equality at all types:
\[ f = g \to fx = gx, \qquad x = y \to fx = fy \]
\item[(iv)] The successor axioms:
\[ \lnot S(x) = 0, \qquad S(x) = S(y) \rightarrow x = y \]
\item[(v)] For any formula $\varphi$ in the language of $\ha$, the induction axiom:
    \[ \varphi(0, \tup y) \to \big( \, \forall x^0 \, ( \, \varphi(x, \tup y) \to \varphi(Sx, \tup y) \, ) \to \forall x^0 \, \varphi(x, \tup y) \, \big). \]
\item[(vi)] The axioms for the combinators:
\begin{eqnarray*}
{\bf k} x y & = & x \\
{\bf s} x y z & = & xz(yz) \\
{\bf p}_0({\bf p}xy) & = & x \\
{\bf p}_1({\bf p}xy) & = & y \\
{\bf p}({\bf p}_0x)({\bf p}_1x) & = & x
\end{eqnarray*} as well as for the recursor:
\begin{eqnarray*}
{\bf R} x y 0 & = & x \\
{\bf R} x y (Sn) & = & yn({\bf R}xyn)
\end{eqnarray*}
\end{enumerate}
This completes the description of the system $\ha$.

\begin{prop}{lambdaabstractionforhaomega} For each term $t$ in the language of $\ha$ and variable $x$, one can construct a term $\lambda x.t$, whose free variables are those of $t$ excluding $x$, such that $\ha \vdash (\lambda x.t) x = t$.
\end{prop}
\begin{proof} As in \refprop{lambdaabstractionforhaomega}.
\end{proof}

Two extensions of $\ha$ will be important in what follows. One is the ``intensional'' variant where for each type $\sigma$ we have a combinator ${\bf e}^\sigma$ of type $\sigma \to (\sigma \to 0)$ satisfying
\[ {\bf e}^\sigma xy \leq 1, \qquad \forall x^\sigma, y^\sigma \, ( \, {\bf e}^\sigma xy = 0 \leftrightarrow x =_\sigma y \, ). \]
This extension of $\ha$ is denoted by $\iha$.

In addition, we have the extensional variant, where we have for all finite types $\sigma$ and $\tau$ the axiom
\[ \forall f, g^{\sigma \to \tau} \, ( \, \forall x^\sigma \, fx =_\tau gx \to f =_{\sigma \to \tau} g \, ). \]
This extension of $\ha$ is denoted by $\eha$.

This section mainly concerns the axiom of choice for all finite types, denoted by $\AC$, which is the following scheme:
\[ \forall x^\sigma \, \exists y^\tau \, \varphi(x, y, \tup z) \to \exists f^{\sigma \to \tau} \, \forall x^\sigma \, \varphi(x, fx, \tup z). \]
Indeed, the purpose of this section is to prove that both $\iha + \AC$ and $\eha + \AC$ are conservative over $\HA$ (Goodman's Theorem). Our strategy will be to show that these systems prove the same arithmetical sentences as $\HAP_\varepsilon$, by giving suitable interpretations of these systems into $\HAP$.

\subsection{Goodman's Theorem} There is a relatively straightforward interpretation of $\ha$ inside $\HAP$. The idea is to define, by induction on the finite type $\sigma$, a predicate ${\rm HRO}_\sigma$ which picks out those elements which are suitable for representing objects of type $\sigma$, as follows:
\begin{eqnarray*}
{\rm HRO}_0(x) & := & x = x \\
{\rm HRO}_{\sigma \times \tau}(x) & := & {\rm HRO}_\sigma({\bf p}_0x) \land {\rm HRO}_\tau({\bf p}_1x) \\
{\rm HRO}_{\sigma \to \tau}(x) & := & \forall y \, ( {\rm HRO}_\sigma(y) \to x \cdot y \downarrow \land \, {\rm HRO}_\tau( x \cdot y) \, )
\end{eqnarray*}
Moreover, equality at all types is interpreted as equality, the combinators ${\bf k}, {\bf s}, {\bf p}, {\bf p}_0, {\bf p}_1$ are interpreted as themselves, while ${\bf r}$ interprets ${\bf R}$ (indeed, in this way any term and any (atomic) formula in the language of $\ha$ can be seen as a term or formula in the language of $\HAP$, by forgetting the types and replacing the combinators in $\ha$ by their analogues in $\HAP$). In addition, we can define a closed term ${\bf e}$ in $\HAP$ such that
\[ {\bf e}xy \leq 1, \quad {\bf e}xy = 0 \leftrightarrow x = y. \]
From this it follows that we can interpret all of $\iha$ inside $\HAP$.

To also interpret $\AC$ we combine this idea with realizability.
\begin{defi}{realizabilityforHAinAPP}  For each $\ha$-formula $\varphi$ with free variables among $y_1^{\sigma_1},\ldots,y_n^{\sigma_n}$ we define a $\HAP$-formula $x \abr \varphi$ (``$x$ realizes $\varphi$'') with free variables among $x, y_1,\ldots,y_n$ by induction on the structure of $\varphi$ as follows:
\begin{eqnarray*}
x \abr \varphi & := & \varphi \mbox{ if } \varphi \mbox{ is atomic} \\
x \abr (\varphi \land \psi) & := & {\bf p}_0x \abr \varphi \land {\bf p}_1x \abr \psi \\
x \abr (\varphi \to \psi) & := & \forall y \, ( y \abr \varphi \to xy \downarrow \land \, xy \abr \psi) \\
x \abr \forall y^\sigma \, \varphi & := & \forall y \, ( {\rm HRO}_\sigma(y) \to \, x \cdot y \downarrow \land \, xy \abr \varphi) \\
x \abr \exists y^\sigma \, \varphi & := & {\rm HRO}^\sigma({\bf p}_0x) \land( {\bf p}_1 x \abr \varphi)[{\bf p}_0x/y]
\end{eqnarray*}
\end{defi}

\begin{theo}{interpretingACintensional} If $\iha + \AC \vdash \varphi(y_1^{\sigma_1},\ldots,y_n^{\sigma_n})$, then
\[ \HAP \vdash \exists x \, \forall y_1,\ldots, y_n \, \big( \, {\rm HRO}_{\sigma_1}(y_1) \land \ldots \land {\rm HRO}_{\sigma_n}(y_n) \to x \cdot y_1 \cdot \ldots \cdot y_n \abr \varphi \, \big). \]
\end{theo}
\begin{proof}
This is proved by induction of the derivation $\varphi(y_1^{\sigma_1},\ldots,y_n^{\sigma_n})$ inside $\iha + \AC$. Note that if  \[ t = \lambda \tup z.\lambda u.{\bf p}(\lambda x^\sigma.{\bf p}_0(ux), \lambda x^\sigma.{\bf p}_1(ux)), \]
then $\HAP \vdash t \abr \AC$.
\end{proof}

\begin{theo}{goodmanstheorem} {\rm (Goodman's Theorem)}
The system $\iha + \AC$ is conservative over $\HA$.
\end{theo}
\begin{proof}
If $\varphi$ is an arithmetical sentence, we have the following sequence of implications:
\begin{displaymath}
\begin{array}{lcr}
\iha + \AC \vdash \varphi & \Longrightarrow & \mbox{(\reftheo{interpretingACintensional}) } \\
\HAP \vdash \exists x \, x \abr \varphi & \Longrightarrow \\
\HAP_\varepsilon \vdash \exists x \, x \abr \varphi & \Longrightarrow & \mbox{(\reftheo{APPepsilonselfrealizing}) } \\
\HAP_\varepsilon \vdash \varphi & \Longrightarrow & \mbox{(\reftheo{HAPepsilonconservative}) } \\
\HA \vdash \varphi. & &
\end{array}
\end{displaymath}
\end{proof}

\subsection{An extensional version of Goodman's Theorem} To interpret $\eha + \AC$ inside $\HAP$ we need to make both HRO and the realizability interpretation from the previous subsection more extensional. As a first step, let us define HEO, an extensional variant of HRO, and indicate how it can be used to interpret $\eha$ inside $\HAP$.

To do this, we still interpret the combinators ${\bf k}, {\bf s}, {\bf p}, {\bf p}_0, {\bf p}_1$ as themselves, while ${\bf r}$ interprets ${\bf R}$. But we need a different criterion for when an object inside $\HAP$ is suitable for representing an object of type $\sigma$; also, we can no longer use equality in $\HAP$ to interpret equality at all finite types. To address this, define by induction on $\sigma$ the following provably symmetric and transitive relations in the language of $\HAP$:
\begin{eqnarray*}
x \sim_0 y & := & x = y \\
x \sim_{\sigma \times \tau} y &:= & {\bf p}_0 x \sim_\sigma {\bf p}_0 y \land {\bf p}_1 x \sim_\tau {\bf p}_1 y \\
x \sim_{\sigma \to \tau} y & := & \forall z, z' \, \big( \, z \sim_\sigma z' \to x \cdot z \downarrow \land x \cdot z' \downarrow \land y \cdot z \downarrow \land y \cdot z' \downarrow \land \\
& & x \cdot z \sim_\tau x \cdot z' \land y \cdot z  \sim_\tau y \cdot z' \land x \cdot z \sim_\tau y \cdot z \, \big).
\end{eqnarray*}
Now define ${\rm HEO}_\sigma(x)$ as $x \sim_\sigma x$. This yields an alternative way of selecting  elements $x$ in $\HAP$ which are suitable for representing objects of type $\sigma$. The point is that if we now interpret equality of objects of type $\sigma$ as $\sim_\sigma$, the result will be an interpretation of $\eha$ inside $\HAP$.

To interpret $\AC$ as well, we have to combine this with a suitably extensional form of realizability.
\begin{defi}{realizabilityforEHAinAPP} For each $\ha$-formula $\varphi$ with free variables among $y_1^{\sigma_1}, \ldots, y_n^{\sigma_n}$, we define new $\HAP$-formulas $x \aber \varphi$ (``$x$ extensionally realizes $\varphi$'') and $x = x' \aber \varphi$ (``$x$ and $x'$ are identical as extensional realizers of $\varphi$'') with free variables among $x, y_1,\ldots,y_n$ by simultaneous induction on the structure of $\varphi$ as follows:
\begin{eqnarray*}
x \aber \varphi & ;= & \varphi \mbox{ if } \varphi \mbox{ is atomic} \\
x = x' \aber \varphi & := & \varphi \land x = x' \\
x \aber (\varphi \land \psi) & := & {\bf p}_0 x \aber \varphi \land {\bf p}_1x \aber \psi \\
x = x' \aber (\varphi \land \psi) & := & {\bf p}_0x = {\bf p}_0x' \aber \varphi \land {\bf p}_1x = {\bf p}_1x' \aber \psi \\
x \aber (\varphi \to \psi) & := & \forall y, y' \, ( \, y = y' \aber \varphi \to xy \downarrow \land xy' \downarrow \land xy = xy' \aber \psi \, ) \\
x = x' \aber (\varphi \to \psi) & := & x \aber (\varphi \to \psi) \land x' \aber (\varphi \to \psi) \land \forall y  \, ( y \aber \varphi \to x y = x'y \aber \psi  \, ) \\
x \aber \forall y^\sigma \, \varphi & := & \forall y, y' \, ( \, y \sim_\sigma y' \to x \cdot y \downarrow \land x \cdot y' \downarrow \land \, x \cdot y = x \cdot y' \aber \varphi \, ) \\
x = x' \aber \forall y^\sigma \, \varphi & := & x \aber \forall y^\sigma \, \varphi \land x' \aber \forall y^\sigma \, \varphi \land \forall y \, ( \, {\rm HEO}_\sigma(y) \to x y = x' y \aber \varphi \, ) \\
x \aber \exists y^\sigma \, \varphi & := & {\rm HEO}_\sigma({\bf p}_0x) \land ({\bf p}_1 x \aber \varphi)[{\bf p}_0x/y] \\
x = x' \aber \exists y^\sigma \, \varphi & := & ({\bf p}_1 x = {\bf p}_1x' \aber \varphi)[{\bf p}_0x/y] \land {\bf p}_0x \sim_\sigma {\bf p}_0x'
\end{eqnarray*}
\end{defi}

\begin{lemm}{basicpropertiesofaber} For any $\ha$-formula $\varphi$ we can prove in $\HAP$ that:
\begin{enumerate}
\item $x \aber \varphi$ is equivalent to $x = x \aber \varphi$;
\item the relation $x = x' \aber \varphi$ is symmetric and transitive;
\item if  $x  \aber \varphi$ and $y \sim_{\sigma} y'$, then $(x  \aber \varphi)[y'/y]$;
\item if $x  = x'\aber \varphi$ and $y \sim_{\sigma} y'$, then $(x = x' \aber \varphi)[y'/y]$.
\end{enumerate}
\end{lemm}
\begin{proof}
The idea is to prove the conjunction of (1 -- 4) by simultaneous induction on the structure of $\varphi$.
\end{proof}

\begin{theo}{interpretingACextensional}
If $\eha + \AC \vdash \varphi(y_1^{\sigma_1}, \ldots, y_n^{\sigma_n})$, then \[ \HAP \vdash \exists x \forall y_1,\ldots,y_n \, \big( \, {\rm HEO}_{\sigma_1}(y_1) \land \ldots \land {\rm HEO}_{\sigma_n}(y_n) \to x \cdot y_1 \cdot \ldots \cdot y_n \aber \varphi \, \big). \]
\end{theo}
\begin{proof}
Again a straightforward induction on the length of the derivation of $\varphi$ in $\eha + \AC$, with \[ t = \lambda \tup z.\lambda u.{\bf p}(\lambda x^\sigma.{\bf p}_0(ux), \lambda x^\sigma.{\bf p}_1(ux)) \]
still realizing $\AC$.
\end{proof}

\begin{theo}{extgoodmanstheorem} {\rm (Beeson's extensional version of Goodman's Theorem)}
The system $\eha + \AC$ is conservative over $\HA$.
\end{theo}
\begin{proof}
If $\varphi$ is an arithmetical sentence, we have the following sequence of implications:
\begin{displaymath}
\begin{array}{lcr}
\eha + \AC \vdash \varphi & \Longrightarrow & \mbox{(\reftheo{interpretingACextensional}) } \\
\HAP \vdash \exists x \, x \aber \varphi & \Longrightarrow \\
\HAP_\varepsilon \vdash \exists x \, x \aber \varphi & \Longrightarrow & \mbox{(\reftheo{APPepsilonselfrealizing}) } \\
\HAP_\varepsilon \vdash \varphi & \Longrightarrow & \mbox{(\reftheo{HAPepsilonconservative}) } \\
\HA \vdash \varphi. & &
\end{array}
\end{displaymath}
\end{proof}

\subsection{Applications to classical systems} Goodman's Theorem has interesting consequences for classical systems as well, as we will now explain. (The results in this subsection were pointed out to us by Ulrich Kohlenbach, answering a question by Fernando Ferreira.)

In what follows we will call a formula in the language of $\ha$ \emph{quantifier-free} if it contains no quantifiers and no equalities of higher type (hence such a formula is built from equalities of type 0 and the propositional operations $\land,\lor,\to$). The \emph{quantifier-free axiom of choice}, denoted by $\QFAC$, is the schema
\[ \forall x^\sigma \, \exists y^\tau \, \varphi(x, y, \tup z) \to \exists f^{\sigma \to \tau} \, \forall x^\sigma \, \varphi(x, fx, \tup z), \]
where $\varphi$ is assumed to be quantifier-free.
\begin{theo}{kohlenbach1} {\rm (Kohlenbach)}
$\ipa + \QFAC$ is conservative over $\PA$.
\end{theo}
\begin{proof} (Compare \cite[Theorem 4.1]{kohlenbach92}.)
Suppose $\varphi$ is an arithmetical sentence provable in $\ipa + \QFAC$. Without loss of generality, we may assume that $\varphi$ is in prenex normal form:
\[ \varphi := \exists x_1 \, \forall y_1 \, \ldots \exists x_n \forall y_n \, \varphi_{qf}(x_1, y_1,\ldots,x_n,y_n), \]
with $\varphi_{qf}$ quantifier-free. If
\[ \ipa + \QFAC \vdash \varphi, \]
then also
\[ \ipa + \QFAC \vdash \varphi^H, \]
where
\[ \varphi^H := \forall f_1,\ldots,f_n \, \exists x_1,\ldots,\exists x_n \, \varphi_{qf}(x_1,f_1(x_1),\ldots,x_n,f_n(x_1,\ldots,x_n)) \]
is the Herbrand normal form of $\varphi$. By combining negative translation and the Dialectica interpretation (that is, the Shoenfield interpretation), it follows that
\[ \iha \vdash \varphi^H \]
as well. But then
\[ \iha \vdash \lnot \exists f_1,\ldots,f_n \, \forall x_1,\ldots,x_n \, \lnot \varphi_{qf}(x_1,f_1(x_1),\ldots,x_n, f_n(x_1,\ldots,x_n)), \]
and therefore
\[ \iha + \AC \vdash \lnot \forall x_1 \, \exists y_1 \ldots \forall x_n \, \exists y_n  \, \lnot \varphi_{qf}(x_1,y_1,\ldots,x_n,y_n). \]
By Goodman's Theorem we obtain
\[ \HA \vdash \lnot \forall x_1 \, \exists y_1 \ldots \forall x_n \, \exists y_n  \, \lnot \varphi_{qf}(x_1,y_1,\ldots,x_n,y_n), \] and therefore $\PA \vdash \varphi$.
\end{proof}

\begin{theo}{kohlenbach2} {\rm (Kohlenbach)}
The system $\epa + \QFAC$ is conservative over $\PA$.
\end{theo}
\begin{proof}
One can formalise the ECF-model of $\epa + \QFAC$ inside $\pa + \QFAC$ (see \cite[Theorem 2.6.20]{troelstra73}). This interpretation does not affect the meaning of statements mentioning only objects of type 0 and $1 = (0 \to 0)$, and therefore $\epa + \QFAC$ is conservative over $\pa + \QFAC$ for statements of this type. So this theorem follows from the previous.
\end{proof}

\bibliographystyle{plain} \bibliography{dst}

\begin{thebibliography}{10}

\bibitem{aczelrathjen01}
P.~Aczel and M.~Rathjen.
\newblock Notes on constructive set theory.
\newblock Technical Report No. 40, Institut Mittag-Leffler, 2000/2001.

\bibitem{beeson79}
M.J. Beeson.
\newblock Goodman's theorem and beyond.
\newblock {\em Pacific J. Math.}, 84(1):1--16, 1979.

\bibitem{beeson85}
M.J. Beeson.
\newblock {\em Foundations of constructive mathematics}, volume~6 of {\em
  Ergebnisse der Mathematik und ihrer Grenzgebiete (3) [Results in Mathematics
  and Related Areas (3)]}.
\newblock Springer-Verlag, Berlin, 1985.
\newblock Metamathematical studies.

\bibitem{coquand13}
T.~Coquand.
\newblock About {G}oodman's theorem.
\newblock {\em Ann. Pure Appl. Logic}, 164(4):437--442, 2013.

\bibitem{feferman75}
S.~Feferman.
\newblock A language and axioms for explicit mathematics.
\newblock In {\em Algebra and logic ({F}ourteenth {S}ummer {R}es. {I}nst.,
  {A}ustral. {M}ath. {S}oc., {M}onash {U}niv., {C}layton, 1974)}, pages
  87--139. Lecture Notes in Math., Vol. 450. Springer, Berlin, 1975.

\bibitem{feferman79}
S.~Feferman.
\newblock Constructive theories of functions and classes.
\newblock In {\em Logic {C}olloquium '78 ({M}ons, 1978)}, volume~97 of {\em
  Stud. Logic Foundations Math.}, pages 159--224. North-Holland, Amsterdam-New
  York, 1979.

\bibitem{goodman76}
N.D. Goodman.
\newblock The theory of the {G}\"odel functionals.
\newblock {\em J. Symbolic Logic}, 41(3):574--582, 1976.

\bibitem{goodman78}
N.D. Goodman.
\newblock Relativized realizability in intuitionistic arithmetic of all finite
  types.
\newblock {\em J. Symbolic Logic}, 43(1):23--44, 1978.

\bibitem{gordeev88}
L.~Gordeev.
\newblock Proof-theoretical analysis: weak systems of functions and classes.
\newblock {\em Ann. Pure Appl. Logic}, 38(1):121, 1988.

\bibitem{hunter08}
J.D. Hunter.
\newblock {\em Higher-order reverse topology}.
\newblock PhD thesis, University of Wisconsin, 2008.

\bibitem{kohlenbach92}
U.~Kohlenbach.
\newblock Remarks on {H}erbrand normal forms and {H}erbrand realizations.
\newblock {\em Arch. Math. Logic}, 31(5):305--317, 1992.

\bibitem{kohlenbach99}
U.~Kohlenbach.
\newblock A note on {G}oodman's theorem.
\newblock {\em Studia Logica}, 63(1):1--5, 1999.

\bibitem{kohlenbach05c}
U.~Kohlenbach.
\newblock Higher order reverse mathematics.
\newblock In {\em Reverse mathematics 2001}, volume~21 of {\em Lect. Notes
  Log.}, pages 281--295. Assoc. Symbol. Logic, La Jolla, CA, 2005.

\bibitem{kohlenbach08}
U.~Kohlenbach.
\newblock {\em Applied proof theory: proof interpretations and their use in
  mathematics}.
\newblock Springer Monographs in Mathematics. Springer-Verlag, Berlin, 2008.

\bibitem{kreisel59}
G.~Kreisel.
\newblock Interpretation of analysis by means of constructive functionals of
  finite types.
\newblock In {\em Constructivity in mathematics: {P}roceedings of the
  colloquium held at {A}msterdam, 1957 (edited by {A}. {H}eyting)}, Studies in
  Logic and the Foundations of Mathematics, pages 101--128. North-Holland
  Publishing Co., Amsterdam, 1959.

\bibitem{martinlof84}
P.~Martin-L{\"o}f.
\newblock {\em Intuitionistic type theory}, volume~1 of {\em Studies in Proof
  Theory. Lecture Notes}.
\newblock Bibliopolis, Naples, 1984.

\bibitem{martinlof09}
P.~Martin-L{\"o}f.
\newblock 100 years of {Z}ermelo's axiom of choice: what was the problem with
  it?
\newblock In {\em Logicism, intuitionism, and formalism}, volume 341 of {\em
  Synth. Libr.}, pages 209--219. Springer, Dordrecht, 2009.

\bibitem{mints75}
G.~E. Mints.
\newblock Finite studies of transfinite deductions.
\newblock {\em Zap. Nau\v cn. Sem. Leningrad. Otdel. Mat. Inst. Steklov.
  (LOMI)}, 49:67--122, 177--178, 1975.
\newblock Theoretical applications of the methods of mathematical logic, I.

\bibitem{vanoosten97}
J.~van Oosten.
\newblock Extensional realizability.
\newblock {\em Ann. Pure Appl. Logic}, 84(3):317--349, 1997.

\bibitem{vanoosten06}
J.~van Oosten.
\newblock A general form of relative recursion.
\newblock {\em Notre Dame J. Formal Logic}, 47(3):311--318 (electronic), 2006.

\bibitem{renardeldelavalette90}
G.R. Renardel~de Lavalette.
\newblock Extended bar induction in applicative theories.
\newblock {\em Ann. Pure Appl. Logic}, 50(2):139--189, 1990.

\bibitem{schweber15}
N.~Schweber.
\newblock Transfinite recursion in higher reverse mathematics.
\newblock {\em J. Symb. Log.}, 80(3):940--969, 2015.

\bibitem{schwichtenbergwainer12}
H.~Schwichtenberg and S.S. Wainer.
\newblock {\em Proofs and computations}.
\newblock Perspectives in Logic. Cambridge University Press, 2012.

\bibitem{slooten14}
L.C.~van Slooten.
\newblock Arithmetical conservativity results, a theory of operations and
  {G}oodman's theorem.
\newblock Master's thesis, Universiteit Utrecht, 2014.

\bibitem{troelstra73}
A.~S. Troelstra, editor.
\newblock {\em Metamathematical investigation of intuitionistic arithmetic and
  analysis}.
\newblock Lecture Notes in Mathematics, Vol. 344. Springer-Verlag, Berlin,
  1973.

\bibitem{troelstra98}
A.~S. Troelstra.
\newblock Realizability.
\newblock In {\em Handbook of proof theory}, volume 137 of {\em Stud. Logic
  Found. Math.}, pages 407--473. North-Holland, Amsterdam, 1998.

\bibitem{troelstravandalen88a}
A.~S. Troelstra and D.~van Dalen.
\newblock {\em Constructivism in mathematics. {V}ol. {I}}, volume 121 of {\em
  Studies in Logic and the Foundations of Mathematics}.
\newblock North-Holland Publishing Co., Amsterdam, 1988.
\newblock An introduction.

\bibitem{troelstravandalen88b}
A.~S. Troelstra and D.~van Dalen.
\newblock {\em Constructivism in mathematics. {V}ol. {II}}, volume 123 of {\em
  Studies in Logic and the Foundations of Mathematics}.
\newblock North-Holland Publishing Co., Amsterdam, 1988.

\end{thebibliography}

\end{document}